\documentclass[12pt]{amsart}
\usepackage{enumerate}
\usepackage{url}
\usepackage{xspace}

\usepackage[margin=1in]{geometry}

\DeclareMathOperator{\lcm}{lcm}

\DeclareMathOperator{\pnt}{\raise 0.5mm \hbox{\large\textbf{.}}}

\newcommand{\note}[2][ ]{}

\newtheorem{theorem}{Theorem}

\newtheorem{proposition}[theorem]{Proposition}

\newtheorem{conjecture}[theorem]{Conjecture}

\theoremstyle{definition}
\newtheorem{remark}[theorem]{Remark}

\usepackage[pdfauthor={Huy Tai Ha and Fabrizio Zanello and Erik Stokes},
            pdftitle={On the density of the odd values of the partition function},
            pdfsubject={commutative algebra},
            pdfkeywords={simplicial complex, matroid, h-vector,
              pure O-sequence},
            pdfproducer={Latex with hyperref},
            pdfcreator={latex->dvips->ps2pdf},
            pdfpagemode=UseNone,
            bookmarksopen=false,
            bookmarksnumbered=true]{hyperref}

\begin{document}
\title[On the density of the odd values of the partition function]{On the density of the odd values\\of the partition function}
\author{Samuel D. Judge, William J. Keith, and Fabrizio Zanello} \address{Department of Mathematical  Sciences\\ Michigan Tech\\ Houghton, MI  49931-1295}
\email{sdjudge@mtu.edu, wjkeith@mtu.edu, zanello@mtu.edu}

\maketitle

\begin{abstract} The purpose of this note is to introduce a new approach to the study of one of the most basic and seemingly intractable problems in partition theory, namely the conjecture that the partition function $p(n)$ is equidistributed modulo 2. 

Our main result will relate the densities, say $\delta_t$, of the odd values of the $t$-multipartition functions $p_t(n)$, for several  integers $t$. In particular, we will show that if $\delta_t>0$  for some $t\in \{5,7,11,13,17,19,23,25\}$, then (assuming it exists) $\delta_1>0$; that is, $p(n)$ itself is odd with positive density. Notice that, currently, the best unconditional result does not even imply that $p(n)$ is odd for $\sqrt{x}$ values of $n\le x$. In general, we conjecture that $\delta_t=1/2$ for all $t$ odd, i.e., that similarly to the case of $p(n)$, all multipartition functions are in fact equidistributed modulo 2.

Our arguments will employ a number of algebraic and analytic methods, ranging from an investigation modulo 2 of some classical Ramanujan identities and several other eta product results, to a unified approach that studies the parity of the Fourier coefficients of a broad class of modular form identities recently introduced by Radu. 
\end{abstract}

\section{Introduction and main results}

Let $p(n)$ denote the number of \emph{partitions} of a nonnegative integer $n$, i.e., the number of ways $n$ can be written as $n=\lambda_1+\dots +\lambda_k$, for integers $\lambda_1\ge \dots \ge \lambda_k\ge 1$. It is well known (see e.g. \cite{Andr}) that the generating function for $p(n)$ is $$\sum_{n=0}^\infty p(n) q^n = \frac{1}{\prod_{i=1}^{\infty}(1-q^i)}.$$

Over the years, an enormous amount of research, which has developed or employed a variety of combinatorial, algebraic or analytic techniques, has been devoted to the problem of studying the behavior of the partition function $p(n)$ modulo a prime number $p$ (see for instance, as a highly nonexhaustive list, \cite{AhlOno, Mahl, Ono1, Radu1, Rama, Serre}). While a significant body of literature is now available on the properties of $p(n)$ modulo $p$ for $p\ge 5$ (i.e., when $p$ does not divide 24), much less is  known today on the behavior of $p(n)$ modulo 2 or 3. The present work focuses on the case $p=2$, studied in print at least since the 1967 paper \cite{PaSh}, which itself mentions some previous interest by Atkins and Dyson.

Even though infinitely many values of $p(n)$ must be odd, we are still very far from knowing whether the partition function is odd with \emph{positive  density}; i.e., if we define 
$$\delta_1 = \lim_{x \rightarrow \infty} \frac{\# \{n \leq x : p(n) {\ }\text{is odd} \}}{x},$$
if the limit exists, then proving that $\delta_1 > 0$ seems to be well beyond the current technology.  In fact, the best asymptotic lower bound known today on the number of odd values of $p(n)$ for $n\le x$, due to Bella\"iche, Green, and Soundararajan \cite{BGS} (who have improved results of several other authors; see, just as a sample, \cite{Ahl,BelNic,Ei,Nic,Ono3,Ser}), is given by $\frac{\sqrt{x}}{\log \log x}$, for $x \rightarrow \infty$. (Tangentially, we know from \cite{BelNic} that the number of $n \leq x$ such that $p(n)$ is even is at least of order $\sqrt{x} \log{\log{x}}$; notice also that, unlike for the odd values, an asymptotic lower bound of order $\sqrt{x}$ for the even values is very easy to obtain.)

It is conjectured that $\delta_1$ exists and equals $1/2$ (see e.g. \cite{Calk,PaSh}), and indeed, it is widely believed that $p(n)$ behaves essentially ``randomly'' modulo 2.

For any positive integer $t$, define $p_t(n)$ as the number of \emph{$t$-multipartitions} of $n$; i.e., the integers $p_t(n)$ are the coefficients of the $t$-th power of the generating function of $p(n)$:
$$\sum_{n=0}^\infty p_t(n) q^n = \frac{1}{\prod_{i=1}^{\infty}(1-q^i)^t}.$$

The ordinary partition function is therefore $p(n) = p_1(n)$.  Likewise, denote by $\delta_t$ the density of the odd values of $p_t(n)$:
$$\delta_t = \lim_{x \rightarrow \infty} \frac{\# \{n \leq x : p_t(n) {\ }\text{is odd} \}}{x},$$
if this limit exists. Similarly to the case of $p(n)$, showing that $\delta_t$ exists and is positive also seems an intractable problem, for any value of $t$. Currently, the best asymptotic lower bound   on the number of odd values of $p_t(n)$ with $n\le x$ is given by $\frac{\sqrt{x}}{\log \log x}$ (see \cite{BGS}; see also the third author \cite{Zan} for a slightly weaker elementary bound  when $t=3$). Hence, notice that a bound of $\sqrt{x}$ is not known yet for any $t\ge 1$.  

In general, extending the conjecture that $\delta_1=1/2$, we believe that $\delta_t=1/2$ for all odd positive values of $t$; in other words, that \emph{all $t$-multipartition functions  are equidistributed modulo 2}. Our conjecture is supported by substantial computer evidence. We have:

\begin{conjecture}\label{mainconj} $\delta_t$ exists and equals $\frac{1}{2}$, for any odd positive integer $t$. Equivalently, if $t = 2^k t_0$ with $t_0\ge 1$ odd, then $\delta_t$ exists and equals $2^{-k-1}$.
\end{conjecture}

(Notice that, because of obvious parity reasons, it suffices to assume that $t$ be odd.) Our main goal is to introduce a new approach to the study of the parity of $p(n)$, by connecting it to the parity of $p_t(n)$, for several  values of $t$. In order to do this, we will employ a number of algebraic and analytic methods described below. The next theorem is  the main result of this paper. It will follow as a corollary of certain partition congruences of independent interest, which we will state in the subsequent theorems.

\begin{theorem}\label{DenseThm} Assume all densities $\delta_i$ below exist. Then $\delta_t>0$ implies $\delta_1>0$ for
$$t=5,7,11,13,17,19,23,25.$$
Moreover, $\delta_t > 0$ implies $\delta_r > 0$ for the following pairs $(t,r)$:
$$(27,9), (9,3), (25,5), (15,3), (21,3), (27,3).$$
\end{theorem}

(Notice that the case $(27,3)$ follows by combining $(27,9)$ and $(9,3)$, and $(25,1)$ follows from $(25,5)$ and $(5,1)$.) In particular, Theorem \ref{DenseThm} says that if the coefficients of the  $t$-th power of the generating function of $p(n)$ are odd with positive density for some $t\in \{5,7,11,13,17,19,23,25\}$, then $p(n)$ itself is odd with positive density (assuming this density exists) --- a fact that, as we saw earlier, seems virtually impossible to show unconditionally today.

In general, the relationship between the density of the odd coefficients of a series in $\mathbb{F}_2[[q]]$ and the density of the odd coefficients of one of its (odd) powers is not obvious. A simple example is given by $F(q)=\prod_{i=0}^\infty (1+q^{4^i})$, whose nonzero coefficients are supported on the numbers with binary digits all in even places (and therefore have density zero), while it is easy to see that, modulo 2, $F(q)^3\equiv \sum_{n=0}^\infty q^n$ (of course giving density 1).

We will be able to prove many of the results of Theorem \ref{DenseThm} by algebraic means, by studying a number of congruences modulo 2, including a simple application of the so-called Ramanujan's ``most beautiful identity,'' i.e., a generating function identity for $p(5n+4)$. This will be done in the next section. We will also show a few other related results of independent interest in that section. For instance, applying the same methods, we will  provide a short proof, and in a sense an explanation, of  the ``striking result'' of a recent paper of Hirschhorn and Sellers \cite{HS}, which in turn improved the seven-author paper \cite{Calkin} --- namely that the odd coefficients of 5-regular partitions have density at most $1/4$ (if this density exists). (In fact, we will relate the parity of 5-regular partitions with that of 20-regular partitions.) Finally, in the last section of the paper, we will show all of the congruences involved in our main theorem, including those that we have not been able to prove algebraically, by a unified approach to the parity of the Fourier coefficients of a broad class of modular forms recently introduced by Radu \cite{Radu}.  In addition to the value often derived from possessing multiple independent proofs of a result by substantially different means, there is a general sense in this field that such algebraic proofs are more satisfying and often give insight that is missing from the usually computer-powered verifications that establish modular form congruences.

Theorem \ref{DenseThm} is a consequence of the next two results, which we will spend the bulk of our paper proving. Recall that two series $f(q) = \sum_{n=n_1}^\infty a(n)q^n$ and $g(q) = \sum_{n=n_2}^\infty b(n) q^n$ are said to satisfy $f(q) \equiv g(q) \pmod{m}$, if $a(n) \equiv b(n) \pmod{m}$ for all integers $n$. All congruences in this work will be modulo 2, unless otherwise stated.

\begin{theorem}\label{TwoTermRels}
The congruence $$q\sum_{n=0}^{\infty}p_t(a n +b)q^n\equiv \frac{1}{\prod_{i=1}^{\infty}(1-q^i)^{at}}+\frac{1}{\prod_{i=1}^{\infty}(1-q^{ai})^t}$$
holds for the following twelve triples $(a,b,t)$:
$$(5,4,1), (7,5,1), (11,6,1), (13,6,1), (17,5,1), (19,4,1),$$$$ (23,1,1), (3,2,3), (5,2,3), (7,1,3), (5,0,5), (3,0,9).$$
\end{theorem}

\begin{theorem}\label{ThreeTermRels}
The congruence
$$q^2\sum_{n=0}^\infty p_t (a^2n+b)q^n\equiv \frac{1}{\prod_{i=1}^{\infty}(1-q^i)^{a^2t}}+\frac{1}{\prod_{i=1}^{\infty}(1-q^{ai})^{at}}+\frac{q}{\prod_{i=1}^{\infty}(1-q^i)^{t}}$$
holds for the following two triples $(a,b,t)$:
$$(3,8,3),(5,24,1).$$
\end{theorem}

The proof of Theorem \ref{DenseThm} is easily deduced once Theorems \ref{TwoTermRels} and \ref{ThreeTermRels} are established, as we explain now for $\delta_5 > 0$ implying $ \delta_1 > 0$; the logic in the other cases is similar.

\begin{proof}
Suppose we know Theorem \ref{TwoTermRels} for case $(5,4,1)$; i.e., 
\begin{equation}\label{FiveFourOne}q \sum_{n=0}^\infty p(5n+4) q^n \equiv  \frac{1}{\prod_{i=1}^{\infty}(1-q^i)^{5}}+\frac{1}{\prod_{i=1}^{\infty}(1-q^{5i})}.
\end{equation}

Assuming that $\delta_5 > 0$ and that $\delta_1$ exists, suppose that $\delta_1=0$.  Thus $\#\{n \leq x: p_5(n) {\ }\text{is odd} \} = \delta_5 x + o(x)$, while the number of odd coefficients up to $x$ of $1/\prod_{i=1}^{\infty}(1-q^{5i})=\sum_{n=0}^\infty p(n) q^{5n}$, namely $\#\{n \leq x/5: p(n){\ }\text{is odd}\}$, is $o(x)$. It easily follows that the odd coefficients up to $x$ of $\sum_{n=0}^\infty p(5n+4) q^n$ are also $\delta_5 x + o(x)$. Hence $\delta_1\ge \delta_5/5>0$, a contradiction.
\end{proof}

Notice that a number of additional results can also follow from these congruences.  Some are likely counterfactual but might serve as hypothesis testing or sources of absurdities in future investigations.  For instance, assuming $\delta_1 = 1$, one would obtain the following.

\begin{proposition}\label{five} If $\delta_1 = 1$, then $\delta_5 = 4/5$, with density zero for the odd coefficients of the series $\sum_{n=0}^\infty p_5(5n) q^{5n}$ and density 1 among all other coefficients. 
\end{proposition}

\begin{proof} If $\delta_1 = 1$, then in  (\ref{FiveFourOne}),  the left-hand side will be odd with density 1, since $5n+4$ is a subprogression of $n$.  On the right-hand side, the coefficients of $q^{5n}$ in the second term will be odd with density 1, and therefore the same coefficients in the first term must be odd with density zero. Outside of this progression, however, the coefficients of the first term must be odd with density 1, immediately giving the desired conclusion.
\end{proof}

Notice that, as seems highly likely to be the case, it follows  that \emph{if the odd density of $p_5(n)$ does not satisfy the conclusions of Proposition \ref{five}, then, should it exist, $\delta_1 < 1$, and thus the {even} coefficients of $p(n)$ have positive density} --- a fact that, as we mentioned earlier, is widely believed to be true but appears to be well beyond the current state of the art.

The converse of Theorem \ref{DenseThm}, i.e., for instance  showing that $\delta_1 > 0$ implies $ \delta_5 > 0$,  is not apparent and seems to require more knowledge of the structure of any odd density of $p(n)$. One can obtain, just as an example, the following weaker result:

\begin{proposition} Assuming that $ \delta_1$ exists, we have that $\delta_5 = 0$ implies $ \delta_1 \leq \frac{5}{6}$. Moreover, in that case, if $\delta_1>0$ then almost all of the odd values of $p(5n+4)$ come from the $25k+24$ subprogression.
\end{proposition}

\begin{proof}
Let $x$ denote the odd density of $p(5n+4)$, and $y$ the odd density of the partition function outside the $5n+4$ progression, if these are well defined.  Assume $\delta_5 = 0$. Thus, by (\ref{FiveFourOne}), it easily follows that $x$ is well defined and has value  $\delta_1/5$, and therefore $y$ is also well defined.  Then clearly, $\delta_1 = 4y/5 + x/5$, which gives us that $\delta_1 =5y/6\le 5/6$. 

Moreover, we have from (\ref{FiveFourOne}) that if $\delta_1 > 0$, $p(25k+24)$ must contain almost all of the odd values of $p(5n+4)$. Indeed, the second term on the right-hand side contributes odd values to a positive proportion of the values of $p(25k+24)$ and even values to all other subprogressions of $5n+4$, while by the assumption that $\delta_5 = 0$, the first term is odd with density zero. This completes the proof.
\end{proof}

\section{Algebraic proofs}

In this section, we  show algebraically the cases $(5,4,1)$, $(7,5,1)$, $(13,6,1)$, $(3,2,3)$, $(5,2,3)$, $(5,0,5)$, and $(3,0,9)$ of Theorem \ref{TwoTermRels}, and both cases $(3,8,3)$ and $(5,24,1)$ of Theorem \ref{ThreeTermRels}. We have not been able to determine algebraic proofs for the cases $(11,6,1)$,  $(17,5,1)$, $(19,4,1)$, $(23,1,1)$, and $(7,1,3)$ of Theorem \ref{TwoTermRels}, so these will only be shown analytically in the next section.

We  employ  the standard $q$-series notation $$\prod_{i=0}^\infty (1-aq^i) = (a;q)_\infty \, ; \quad (q;q)_\infty = (q)_\infty.$$

First we prove Theorem \ref{TwoTermRels} for $(5,4,1)$, which is perhaps the easiest case. 

\begin{proof} 
Ramanujan's ``most beautiful identity'' states that 
\begin{equation}\label{bea}
\sum_{n=0}^\infty p(5n+4)q^n = 5 \frac{(q^5)_\infty^5}{(q)_\infty^6}.
\end{equation}

Now, using a congruence of Blecksmith-Brillhart-Gerst (see \cite{BBG}, p. 301; cf. also Hirschhorn's equation (13) in \cite{Hirsch1}), we can easily obtain
\begin{equation}\label{H1} \frac{(q)_\infty}{\prod_{i=1}^\infty (1-q^{10i-5})} \equiv (q)_\infty (q^5)_\infty \equiv \sum_{n=1}^\infty q^{n^2-n}+\sum_{n=1}^\infty q^{5n^2-5n+1}.\end{equation}

It is well known that
$$(q)_\infty^3 = \sum_{n=0}^\infty (-1)^n (2n+1) q^{n(n+1)/2} \equiv \sum_{n=0}^\infty q^{n(n+1)/2},$$
and therefore we may transform equation (\ref{H1}) into
\begin{equation}\label{15} (q)_\infty (q^5)_\infty \equiv (q)_\infty^6 + q (q^{5})_\infty^6.
\end{equation}

It now follows by (\ref{bea}) and standard algebraic manipulations that
\begin{equation}
q \sum_{n=0}^\infty p(5n+4)q^n \equiv q \frac{(q^5)_\infty^5}{(q)_\infty^6} \equiv \frac{1}{(q)_\infty^5} + \frac{1}{(q^5)_\infty},
\end{equation}
which is the desired result.
\end{proof}

As a further and simple application of congruence (\ref{15}), we next give a quick proof of the main result of Hirschhorn-Sellers \cite{HS}, which in turn improved the work of \cite{Calkin} on 5-regular partitions. Recall that a partition of $n$ is \emph{$m$-regular} if it does not contain parts that are multiples of $m$. One  usually denotes the number of $m$-regular partitions of $n$ by $b_m(n)$, and it easily follows from the definition that the generating function for $b_m(n)$ is given by:
$$\sum_{n=0}^\infty b_m(n)q^n=\frac{(q^m)_\infty}{(q)_\infty}.$$

Denote by $\delta^{[m]}$ the density of the odd values of $b_m$, if this density exists. While it is known that, for certain $m$, $\delta^{[m]}=0$ (for instance, when $m=2^a\cdot m_0$ with $m_0\le \sqrt{m}$; see Ono-Gordon \cite{OnoG} and also Serre \cite{Serre0}), for other values of $m$ it seems reasonable to believe that $\delta^{[m]}>0$, though there exists no $m$ yet for which this has been established. (In fact, we simply remark here that if $\delta^{[m]}>0$ for some $m$, then it is easy to see that, asymptotically, $p(n)$ is odd for at least the order of $\sqrt{x}$ values of $n\le x$ --- a fact that, as we saw, is also not known.)

Currently, the best bound for $\delta^{[5]}$ is the ``striking result'' $\delta^{[5]}\le 1/4$, obtained in \cite{HS} (requiring the additional assumption, not explicitly stated in \cite{HS}, that this density exists). The next theorem is in a sense an explanation of this fact, in a way that also nicely relates it to the density of 20-regular partitions.

\begin{theorem}
Assuming existence, we have $\delta^{[5]}=\delta^{[20]}/4$; in particular, $\delta^{[5]} \le 1/4$.
\end{theorem}

\begin{proof}
By multiplying both sides of (\ref{15}) by $1/(q)_\infty^2$, we obtain
\begin{equation}\label{200}
\sum_{n=0}^\infty b_5(n)q^n \equiv (q)_\infty^4 + q \frac{(q^{5})_\infty^6}{(q)_\infty^2}.
\end{equation}

Similarly,  multiplication by $(q^5)_\infty^2/(q)_\infty^2$ across  (\ref{15}) yields
\begin{equation}\label{201}
\frac{(q^{5})_\infty^3}{(q)_\infty}\equiv (q)_\infty^4(q^5)_\infty^2 + q \frac{(q^5)_\infty^8}{(q)_\infty^2}.
\end{equation}

Thus, combining (\ref{200}) and  (\ref{201}), we  get
\begin{equation}\label{202}
\sum_{n=0}^\infty b_5(n)q^n\equiv  (q)_\infty^4 +q(q)_\infty^8(q^5)_\infty^4  + q^3\frac{(q^5)_\infty^{16}}{(q)_\infty^4}\equiv (q)_\infty^4 +q(q)_\infty^8(q^5)_\infty^4  + \sum_{n=0}^\infty b_{20}(n)q^{4n+3}.
\end{equation}

Hence, one sees that the theorem follows if we show that the odd coefficients of $(q)_\infty^4 +q(q)_\infty^8(q^5)_\infty^4 $ have density zero. This  is clearly the case for the first summand, since by the Pentagonal Number Theorem,
$$(q)_\infty^4 \equiv \sum_{n\in \mathbb{Z}} q^{2n(3n-1)}.$$

As for $q(q)_\infty^8(q^5)_\infty^4 $, again by the Pentagonal Number Theorem we have
$$q(q)_\infty^8(q^5)_\infty^4 \equiv q \sum_{m\in \mathbb{Z}} q^{4m(3m-1)} \sum_{n\in \mathbb{Z}} q^{10n(3n-1)}.$$

But it is a classical result of Landau (see \cite{Land} or for instance Serre \cite{Serre}) that the integers that can be represented by a quadratic form in two variables have density zero. Since this obviously remains true once we reduce modulo 2, the proof of the theorem is complete.
\end{proof} 

\begin{remark}
\begin{enumerate}
\item Computer calculations suggest that it seems reasonable to conjecture that $\delta^{[5]}= 1/8$ (or equivalently, that $\delta^{[20]}=1/2$).
\item We  remark here without proof that similar nice applications to $m$-regular partitions can also be given for other values of $m$. For instance, from identity (\ref{Lin}) below, it is possible to deduce in an analogous fashion that, assuming existence, $\delta^{[7]}=\delta^{[28]}/2$ (and, consequently, that $\delta^{[7]}\le 1/2$).
\end{enumerate}
\end{remark}

We now show the remaining algebraic cases. 

\begin{proof} 
We begin with Theorem \ref{TwoTermRels} for $(13,6,1)$. Zuckerman's classical identity for $p(13n+6)$ \cite{Zuckerman} easily implies that
\begin{equation}\label{zuc}
\sum_{n=0}^\infty p(13n+6)q^n \equiv \frac{(q^{13})_\infty}{(q)_\infty^2} + q^5 \frac{(q^{13})_\infty^{11}}{(q)_\infty^{12}}+q^6\frac{(q^{13})_\infty^{13}}{(q)_\infty^{14}},
\end{equation}
while from Calkin \emph{et al.} \cite{Calkin}, we can see that
$$\frac{(q^{13})_\infty}{(q)_\infty}+(q)_\infty^{12} \equiv q (q)_\infty^{10}(q^{13})_\infty^2 + q^6(q^{13})_\infty^{12}+q^7 \frac{(q^{13})_\infty^{14}}{(q)_\infty^2}.$$ 

Thus, now multiply both sides of the latter congruence by
$$\frac{1}{(q)_\infty^{12}(q^{13})_\infty},$$
and then substitute into (\ref{zuc}). This yields the desired result. 

Case $(7,5,1)$ can be proved using the next identity, which is essentially due to Ramanujan (see \cite{BerndtOno}, equation (24.6a)):
$$\sum_{n=0}^\infty p(7n+5)q^n \equiv \frac{(q^7)_\infty^3}{(q)_\infty^4}+q\frac{(q^7)_\infty^7}{(q)_\infty^8},$$ 
along with the following, which is equivalent to an identity of Lin (cf. \cite{Lin}, equation 2.4):
\begin{equation}\label{Lin}
(q)_\infty (q^7)_\infty \equiv (q)_\infty^8 + q(q)_\infty^4 (q^7)_\infty^4 + q^2 (q^7)_\infty^8.
\end{equation}

Divide now through identity (\ref{Lin}) by $(q)_\infty^8 (q^7)_\infty$, and substitute into the previous one to complete the proof.

Case $(3,2,3)$ can be shown in a similar fashion, starting from an identity of  Chan (\cite{chan}, Theorem 1; see also  Xiong \cite{Xiong}, Theorem 1.1), which has the immediate corollary
\begin{equation}\label{xi}
\sum_{n=0}^\infty p_3(3n+2)q^n \equiv \frac{(q^3)_\infty^9}{(q)_\infty^{12}}.
\end{equation}
 
We can combine this  with \cite{HirschSell}, Theorem 2.1 (note that a power of 2 is missing from the factor $(q^{12})_\infty$ in the published version of this paper), which yields
\begin{equation}\label{HSEqn}
\frac{1}{(q)_\infty^9(q^3)_\infty^9}\equiv \frac{q}{(q)_\infty^{12}}+\frac{1}{(q^3)_\infty^{12}}.
\end{equation}

Indeed, if we multiply across this latter by  $(q^3)_\infty^9$, we easily obtain the desired result.

Case $(5,2,3)$ of Theorem \ref{TwoTermRels} is the claim that
$$q \sum_{n=0}^\infty p_3(5n+2) q^n \equiv \frac{1}{(q)_\infty^{15}} + \frac{1}{(q^5)_\infty^3}.$$

We begin with an identity of Chan-Lewis (\cite{ccll}, identity (1.11); see also Xiong \cite{Xiong2}), which modulo 2 becomes
$$ q \sum_{n=0}^\infty p_3(5n+2) q^n \equiv q \frac{(q^5)_\infty^3}{(q)_\infty^6}  + q^2 \frac{(q^5)_\infty^9}{(q)_\infty^{12}}  + q^3 \frac{(q^5)_\infty^{15}}{(q)_\infty^{18}}.$$

The third term on the right-hand side is the cube of $q \sum_{n=0}^\infty p(5n+4)q^n$, which by the $(5,4,1)$ case of Theorem \ref{TwoTermRels}, is congruent to $\frac{1}{(q)_\infty^5} + \frac{1}{(q^5)_\infty}$.  The middle term is $\frac{1}{(q^5)_\infty}$ times the square of $q \sum_{n=0}^\infty p(5n+4)q^n$, and the first term is $\frac{1}{(q^5)_\infty^2}$ times $q \sum_{n=0}^\infty p(5n+4)q^n$ itself, again using Theorem \ref{TwoTermRels} for $(5,4,1)$.

Therefore, now expand the necessary powers of $\frac{1}{(q)_\infty^5} + \frac{1}{(q^5)_\infty}$ and cancel the even terms; the remaining terms easily constitute the desired identity.


Case $(5,0,5)$ of Theorem \ref{TwoTermRels} requires the $(5,24,1)$ case of Theorem \ref{ThreeTermRels}, so we prove the latter first. This is the claim that
$$q^2 \sum_{n=0}^\infty p(25n+24) q^n \equiv \frac{1}{(q)_\infty^{25}} + \frac{1}{(q^5)_\infty^5} + \frac{q}{(q)_\infty}.$$ 

We will show it using Zuckerman's identity \cite{Zuckerman} for $\sum_{n=0}^\infty p(25n+24)q^n$ (or alternatively, \cite{BerndtOno}, equation (21.1)), which, modulo 2, gives us that
\begin{equation}\label{zzuu}
\sum_{n=0}^\infty p(25n+24)q^n \equiv \frac{(q^5)_\infty^6}{(q)_\infty^7}+q^2 \frac{(q^5)_\infty^{18}}{(q)_\infty^{19}}+q^4\frac{(q^5)_\infty^{30}}{(q)_\infty^{31}},
\end{equation}
along with a repeated application of our recently proven fact (see the case $(5,4,1)$):
\begin{equation}\label{5i}
q \frac{(q^5)_\infty^5}{(q)_\infty^6} \equiv \frac{1}{(q)_\infty^5} + \frac{1}{(q^5)_\infty}.
\end{equation}

Indeed, let us multiply through (\ref{zzuu}) by $q^2$, then treat the third term on the right-hand side as $(q)_\infty^5$ times the 6th power of (\ref{5i}), the second term as $(q)_\infty^5/(q^5)_\infty^2$ times the 4th power of (\ref{5i}), and the first term as $q (q^5)_\infty/(q)_\infty$ times (\ref{5i}) itself. Finally, cancel what terms are possible to eventually obtain:
$$q^2 \sum_{n=0}^\infty p(25n+24) q^n \equiv \frac{q}{(q)_\infty} + \frac{1}{(q)_\infty^{25}} + q \frac{(q^5)_\infty}{(q)_\infty^6} + \frac{1}{(q)_\infty^5 (q^5)_\infty^4}.$$

If now  we  again apply (\ref{5i})  (dividing by $ (q^5)_\infty^4$) to the third term on the right-hand side of (\ref{zzuu}) and then cancel, we finally obtain the desired congruence. 

The $(5,0,5)$ case of Theorem \ref{TwoTermRels} is the claim that
$$q \sum_{n=0}^\infty p_5(5n) q^n \equiv \frac{1}{(q)_\infty^{25}} + \frac{1}{(q^5)_\infty^5}.$$

In order to prove this, use the $(5,24,1)$ identity and divide through by $q$, to transform the claim into
\begin{equation}\label{2524}
\sum_{n=0}^\infty p_5(5n)q^n \equiv q \sum_{n=0}^\infty p(25n+24) q^n + \frac{1}{(q)_\infty}.
\end{equation}

Recall from the case $(5,4,1)$ that
$$q \sum_{n=0}^\infty p(5n+4) q^n \equiv \frac{1}{(q)_\infty^5} + \frac{1}{(q^5)_\infty}.$$  

Thus, now extract every power $q^{5n}$ from this identity, and equate the resulting series modulo 2. We obtain
$$q \sum_{n=0}^\infty p(25n+24) q^{5n+4} \equiv \sum_{n=0}^\infty p(n) q^{5n} + \sum_{n=0}^\infty p_5(5n)q^{5n}.$$

Making the substitution $q^5 \rightarrow q$, we then easily get:
$$q \sum_{n=0}^\infty p(25n+24)q^n \equiv \sum_{n=0}^\infty p(n)q^n + \sum_{n=0}^\infty p_5(5n) q^n.$$

This is equivalent to (\ref{2524}), thus completing the proof of  case $(5,0,5)$.

Case $(3,8,3)$ of Theorem \ref{ThreeTermRels} can be proved by 3-dissecting  (\ref{xi}).  Indeed, if we apply  (\ref{xi}) to itself and then expand the $q^{3n+2}$ terms of the denominator, we get 
$$\sum_{n=0}^\infty p_3(9n+8)\equiv q^2 \frac{(q^3)_\infty^{36}}{(q)_\infty^{39}}$$
(see also \cite{xio}, Theorem 1.2). Then, by a repeated application of (\ref{HSEqn}) (the first time by substituting $q \rightarrow q^4$), standard algebraic manipulations yield
$$q^2 \sum_{n=0}^\infty p_3(9n+8) q^n \equiv q^4 \frac{(q^3)_\infty^{36}}{(q)_\infty^{39}} \equiv (q^3)_\infty^{36} (q)_\infty^9 \left( \frac{q^4}{(q^4)_\infty^{12}} \right) $$$$ \equiv (q^3)_\infty^{36} (q)_\infty^9 \left( \frac{1}{(q^{12})_\infty^{12}} + \frac{1}{(q^4)_\infty^9 (q^{12})_\infty^9} \right) $$$$ \equiv \frac{1}{(q)_\infty^{27}} + \frac{(q)_\infty^9}{(q^3)_\infty^{12}} \equiv  \frac{1}{(q)_\infty^{27}} + \frac{1}{(q^3)_\infty^9} + \frac{q}{(q)_\infty^3},$$
which is the desired congruence for $(3,8,3)$.

Finally, case $(3,0,9)$ is proved similarly to case $(5,0,5)$, by combining our result for $(3,8,3)$ above with that for $(3,2,3)$, so we will omit the details.
\end{proof}

\section{Modular form proofs}

In this final section, we prove the cases $(11,6,1)$, $(17,5,1)$, $(19,4,1)$, $(23,1,1)$, and $(7,1,3)$ of Theorem \ref{TwoTermRels}. All of these --- and, in fact, all cases proved in this paper --- can be shown in a unified fashion using the machinery introduced by Radu in \cite{Radu} to construct and verify a new class of modular form identities of the Ramanujan-Kolberg type.

Our general strategy is as follows (please see below for the relevant definitions): We first determine an $\eta$-quotient which, when multiplied by $\sum_{n=0}^\infty p(an+b)$, yields a modular function, i.e., a modular form of weight zero.  (Whether such an $\eta$-quotient exists is a separate question; should there exist a vector which satisfies certain conditions, Radu's theorem implies that this $\eta$-quotient does exist and has exponents given by the entries of the vector.)  These forms are weakly holomorphic; we rid ourselves of the poles by multiplying by a cusp form of sufficiently high order.  We then apply a classical theorem of Sturm (see \cite{sturm}) to verify the conjectured congruence.

We first require a number of standard facts on modular forms, which we  briefly recap here without proof for the reader's convenience. For background and proofs, we refer for instance to \cite{Koblitz,Ono}.

We  consider modular forms $f$ of weight $k$ and character $\chi$ for $\Gamma_0(N)$, where this latter is defined as the subgroup of SL$(2,\mathbb{Z})$ of those matrices $\left( \begin{matrix} a & b \\ c & d \end{matrix} \right)$ such that $ c \equiv 0 \pmod{N}$, and for $\Gamma_1(N)$, which is the subgroup of $\Gamma_0(N)$ requiring the additional assumption that $a\equiv d\equiv 1$ (mod $N$). The \emph{level} of $f$ is then the least value of $N$ for which $f$ is a modular form of weight $k$ for $\Gamma_0(N)$.  Note that in such a case, $f$ is also a modular form of weight $k$ for any $\Gamma_0(cN)$, where $c \in \mathbb{N^{+}}$.

An \emph{$\eta$-quotient} is a quotient of functions of the form
$$\eta(\delta z) = q^{\frac{\delta}{24}} \prod_{i=1}^\infty (1-q^{\delta i}),$$
where $q = e^{2\pi i z}.$ We recall that an $\eta$-quotient is a modular form under the following conditions, given by a theorem of Gordon-Hughes \cite{GH} and Newman \cite{Newman}.

\begin{theorem} [\cite{GH,Newman}]\label{ghn} Let $f(z) = \prod_{\delta \vert N} \eta^{r_\delta} (\delta z)$, with $r_\delta \in \mathbb{Z}$.  If
$$\sum_{\delta \vert N} \delta r_\delta \equiv 0 \, \pmod{24} \quad \text{ and } \quad \sum_{\delta \vert N} \frac{N}{\delta} r_\delta \equiv 0 \, \pmod{24},$$
then $f(z)$ is a modular form for $\Gamma_0(N)$ of weight $k=\frac{1}{2} \sum r_\delta$ and character $\chi(d) = \left( \frac{(-1)^k s}{d} \right)$, where $s = \prod_{\delta \vert N} \delta^{r_\delta}$ and $\left(\frac{\cdot}{d}\right)$ denotes the Jacobi symbol.
\end{theorem}

Given two modular forms of weight $k$ for $\Gamma_0(N)$, the following crucial result of Sturm  \cite{sturm} gives a criterion for determining when all of their coefficients are congruent modulo a given prime.

\begin{theorem} [\cite{sturm}]\label{sturmthe} Let $p$ be a prime number, and $f(z) = \sum_{n=n_0}^\infty a(n) q^n$ and $g(z) = \sum_{n=n_1}^\infty b(n) q^n$  be holomorphic modular forms of weight $k$ for $\Gamma_0(N)$ of characters $\chi$ and $\psi$, where $n_0,n_1\in \mathbb{N}$.  If either $\chi=\psi$ and 
$$a(n) \equiv b(n) \pmod{p} \quad \text{for all} \quad n\le \frac{kN}{12}\cdot \prod_{d{\ }\emph{prime};{\ } d \vert N} \left(1+\frac{1}{d}\right),$$
or $\chi \neq \psi$ and
$$a(n) \equiv b(n) \pmod{p} \quad \text{for all} \quad n\le \frac{kN^2}{12}\cdot \prod_{d{\ }\emph{prime};{\ } d \vert N} \left(1-\frac{1}{d^2}\right),$$
then $f(z) \equiv g(z) \pmod{p}$ (i.e., $a(n) \equiv b(n) \pmod{p}$ for all $n$).
\end{theorem}

Our primary tool is to suitably apply Radu \cite{Radu}, Theorem 45.  We restate it here to keep the paper self-contained, which will require the following notation and definitions.

Let $R(N)$ denote the set of integer sequences with entries $r_\delta$ indexed by the positive divisors of $N$.  For $r=(r_\delta)_{\delta \vert N} \in R(N)$, define
$$w(r) = \sum_{\delta \vert N} r_\delta \, ; \quad \sigma_\infty (r) = \sum_{\delta \vert N} \delta r_\delta \, ; \quad \sigma_0(r) = \sum_{\delta \vert N} \frac{N}{\delta} r_\delta \, ; \quad \Pi(r) = \prod_{\delta \vert N} \delta^{\vert r_\delta \vert}.$$

Now denote by $\Delta^*$ the set of tuples $(m,M,N,t,(r_\delta)) \in (\mathbb{N^{+}})^3  \times \mathbb{N} \times R(M)$ such that the following conditions are simultaneously satisfied: For all primes $p$, if $p\vert m$ then $p\vert N$; if $\delta \vert M$ and $r_\delta \neq 0$, then $\delta \vert mN$; $t \in \{0,\dots , m-1\}$; finally, if $\kappa = \gcd(1-m^2,24) \geq 1$, then
$$24 {\ }\vert {\ } \frac{\kappa mN^2 \sigma_0(r)}{M}; \quad 8 {\ }\vert {\ } \kappa N w(r); \quad \frac{24m}{\gcd(\kappa (-24t-\sigma_\infty(r)),24m)} {\ }\vert {\ } N;$$
and if $m$ is even and $\prod_{\delta \vert M} \delta^{\vert r_\delta \vert} = 2^s j$ with $j$ odd, then either $4 \vert \kappa N$ and $8 \vert Ns$, or $s$ is even and $ 8 \vert N(1-j)$.  (This latter condition is included for completeness but never applies in this paper.)

Let $g_{m,t}(q)$ be the \emph{$(mn+t)$-dissection} of the $\eta$-quotient defined by the exponents $r_\delta$: That is, for a given $r=(r_\delta) \in R(M)$, let $\sum_{n=0}^\infty a_r(n) q^n = \prod_{\delta \vert M} \prod_{i=1}^\infty (1-q^{\delta i})^{r_\delta}$, and define
$$g_{m,t} (q) = q^{\frac{24t+\sigma_\infty (r)}{24m}} \sum_{n=0}^\infty a_r(mn+t)q^n.$$

For $m,M \in \mathbb{N}^{+}$, now let
$$P_{m,r}(t) = \left\{ \left[ ta^2 + \frac{a^2-1}{24} \sigma_\infty (r) \right] _m : \gamma = \left( \begin{matrix} a & b \\ c & d \end{matrix} \right) \in \Gamma_0(M), a\equiv 1, 5 \pmod{6} \right\},$$
where $[n]_m$ denotes the least nonnegative integer congruent to $n \pmod{m}$.  Finally, set
$$\chi_{m,r}(t) = \prod_{u \in P_{m,r}(t)} \exp((1-m^2)(24u+\sigma_\infty(r))/(24m)),$$
where as usual $\exp(x) = e^{2\pi i x}$. Given these definitions,  Radu's theorem is as follows.

\begin{theorem} [\cite{Radu}]\label{RaduThm}
Let $(m,M,N,t,r=(r_\delta)) \in \Delta^*$, $s=(s_\delta) \in R(N)$, and $\nu$ an integer such that $\chi_{m,r}(t) = \exp(\nu/24)$.  (Such a $\nu$ exists by Lemma 43, \cite{Radu}.) Then $$F(s,r,m,t) (z) = \prod_{\delta \vert N} \eta^{s_\delta} (\delta z) \prod_{u \in P_{m,r}(t)} g_{m,u}(q)$$
is a modular form of weight zero with trivial character for $\Gamma_0(N)$ if and only if the following conditions hold:
\begin{align}
\vert P_{m,r} (t) \vert \cdot w(r) + w(s) = 0; \\
 \nu + \vert P_{m,r} (t) \vert \cdot m \sigma_\infty (r) + \sigma_\infty (s) \equiv 0 \pmod{24}; \\
\vert P_{m,r} (t) \vert \cdot \frac{mN\sigma_0(r)}{M} + \sigma_0(s) \equiv 0 \pmod{24}; \\
\left(\prod_{\delta \vert M} (m \delta)^{\vert r_\delta \vert } \right)^{\vert P_{m,r} (t) \vert }\cdot  \Pi(s) \, \text{ is a square.}
\end{align}

\end{theorem}

We are now ready to  prove the remaining congruences left open in the previous section.

\begin{proof} 
We  show in detail the cases $(11,6,1)$ and $(7,1,3)$ of Theorem \ref{TwoTermRels}.  The other congruences are then simply a matter of finding the appropriate $\eta$-quotients and necessary multipliers to absorb poles, which we will list at the end for the reader's convenience.

We begin with case $(11,6,1)$.  We have $m=11$, $M=1$, $N=22$, $t=6$, and $(r_1)=(-1)$. We verify that this $(m,M,N,t,r)$-vector exists in $\Delta^*$. 

Consider now $ta^2+\frac{a^2-1}{24}\sigma_{\infty}(r)\pmod{m}$, i.e., $6a^2-\frac{a^2-1}{24} \pmod{11}$. Since $24^{-1}=6$  in the field $\mathbb{F}_{11}$, we  have that this expression is identically 6 for all $a$. Thus, $P_{m,r} (t) = \{6\}$.

Now we move to the conditions in Theorem \ref{RaduThm}. We notice that it suffices to choose $\nu=24$, and that we need a vector $s$ such that $\omega(s)=1$, $\sigma_{\infty}\equiv 11 \pmod{24}$, $\sigma_o\equiv 2 \pmod{24}$, and such that  $11 \cdot \Pi(s)$ is a square. One can compute that $s=(s_1,s_2,s_{11},s_{22})=(10,2,11,-22)$ satisfies the above conditions (here the divisors indexing the entries are the positive divisors of 22 in increasing order: $1, 2, 11, 22$). 

Therefore, we can construct $$F(s,r,11,6)=q^{\frac{13}{24}}\frac{\eta(z)^{10}\eta(2z)^2\eta(11z)^{11}}{\eta(22z)^{22}} \sum_{n=0}^\infty p(11n+6)q^n.$$

Recall that the congruence that we want to show is
$$q\sum_{n=0}^\infty p(11n+6)q^n \equiv \frac{1}{(q)_\infty^{11}}+\frac{1}{(q^{11})_\infty},$$
which, in terms of $\eta$-products,  becomes
\begin{equation}\label{116}
q\sum_{n=0}^\infty p(11n+6)q^n \equiv \frac{q^{\frac{11}{24}}}{\eta(z)^{11}}+\frac{q^{\frac{11}{24}}}{\eta(11z)}.
\end{equation}

Multiplying through (\ref{116}) by the $\eta$-quotient required to construct $F$, on the left-hand side we obtain 
$$q^{\frac{13}{24}}\frac{\eta(z)^{10}\eta(2z)^2\eta(11z)^{11}}{\eta(22z)^{22}}\sum_{n=0}^\infty p(11n+6)q^n.$$

By Theorem \ref{RaduThm}, this is a modular form of weight zero. Similarly, the right-hand side becomes
$$\frac{\eta(z)^{10}\eta(2z)^2\eta(11z)^{11}}{\eta(22z)^{22}}\left( \frac{1}{\eta(z)^{11}}+\frac{1}{\eta(11z)}\right ) \equiv\frac{\eta(2z)^2\eta(11z)^{11}}{\eta(z)\eta(22z)^{22}}+\frac{\eta(z)^{10}\eta(2z)^2\eta(11z)^{10}}{\eta(22z)^{22}}$$
$$\equiv\frac{\eta(4z)\eta(11z)^{11}}{\eta(z)\eta(44z)^{11}}+\frac{\eta(z)^{10}\eta(2z)^2\eta(11z)^{10}}{\eta(22z)^{22}},$$
where in the last step we have used the fact that, modulo 2, $\eta(z)^2\equiv \eta(2z)$.

By Theorem \ref{ghn}, both of these terms are weight zero modular forms.  The sum of two such modular forms is also a modular form of weight zero with trivial character (possibly for $\Gamma_0(N)$ with $N$ the least common multiple of the $N$ for each form separately).  We cannot yet apply Sturm's theorem since the forms have poles at infinity (the $q$-expansion yields terms of order $q^{-15}$), and possibly other cusps. (Note that all latter cusps can be represented by rational numbers of the form $\frac{c}{d}$, where $d \vert N$ and $\gcd(c,d)=1$; see e.g. \cite{Ono}.) However, the product of two modular forms of weight $k$ for $\Gamma_0(N)$ and weight $\ell$ for $\Gamma_0(L)$ is a modular form of weight $k+\ell$ for $\Gamma_0(\lcm(N,L))$, and so by Theorem \ref{ghn} we can multiply both sides by a power of $\eta(4z)^{24} = q^4 \prod_{i=1}^\infty (1-q^{4i})^{24}$, or another such form of sufficient power to remove all poles.

In order to determine the minimum order necessary for the right-hand side, we employ a theorem of Ligozat on $\eta$-quotients:

\begin{theorem}[\cite{Ligo,Ono}]\label{ligozat} Let $c$, $d$ and $N$ be integers such that $d \vert N$ and $\gcd(c,d)=1$.  Then if $f$ is an $\eta$-quotient satisfying the conditions of Theorem \ref{ghn} for $\Gamma_0(N)$, the order of $f$ at the cusp $\frac{c}{d}$ is
$$\frac{N}{24} \sum_{\delta \vert N} \frac{\left(\gcd(d,\delta)\right)^2 r_\delta}{\gcd \left(d,\frac{N}{d}\right) d \delta}.$$
\end{theorem}

For the left-hand side we invoke Theorem 47 of \cite{Radu}, which provides a lower bound on the order of $F(s,r,m,t)$ at any such cusp:

\begin{theorem}[\cite{Radu}, Theorem 47, and Eqns. (56-57)]\label{RaduThm47}
For $F(s,r,m,t)$ as constructed above, the order of $F$ at any cusp of $\Gamma_0(N)$ is uniformly  bounded from below by
$$ {{\emph{min}} \atop {c \vert N}} \frac{N}{\gcd(c^2,N)} \left( \vert P_{m,r}(t) \vert {{\emph{min}} \atop {{d \vert m} \atop {\gcd(d,c)=1}}} \frac{1}{24} \sum_{\delta \vert M} r_\delta \frac{(\gcd(\delta d,mc))^2}{\delta m} + \frac{1}{24} \sum_{\delta \vert N} s_\delta \frac{(\gcd(\delta,c))^2}{\delta} \right).$$
\end{theorem}

Note that, in all of our formulas, $\vert P_{m,r}(t) \vert = 1$.  

In the case of $(a,b,t)=(11,6,1)$, the minimum order appearing when we employ the above theorems with $N=44$ is $-15$, and therefore we must remove poles.  Multiplying by increasing powers of $\eta(4z)^{24j}$, thereby increasing the corresponding entries of $r_\delta$ and $s_\delta$, we consider the claim as a congruence of forms in level $N=4m$ and find that, in order to remove all the poles, it is sufficient to multiply on both sides by $\eta(4z)^{360}$, which has weight 180, level 4, and trivial character.  

Now we can apply the Sturm bound of Theorem \ref{sturmthe}. The previous modular forms were both of weight zero and level dividing 44, so the resulting claim is a congruence of holomorphic modular forms of weight 180 and level 44 with trivial character, for which the Sturm bound is 1080.  One can quickly check, e.g. in  Mathematica, that the coefficients on both sides up to $q^{1080}$ are indeed congruent modulo 2. Therefore, the bound is satisfied and (\ref{116}) is proved.

The idea to show the case $(7,1,3)$ is similar and we will only sketch the argument here. Notice that now it suffices to choose $m=7$, $M=1$, $N=14$, $t=1$, and $(r_1)=(-3)$. Again, selecting $\nu=24$ meets the requirements set forth, and we easily have that $P_{m,r}(t)=\{1\}$. 

Picking $s=(10,10,5,-22)$, we see  that we want to multiply the desired congruence for $(7,1,3)$  by $$q^{-\frac{21}{24}}\frac{\eta(z)^{10}\eta(2z)^{10}\eta(7z)^5}{\eta(14z)^{22}}.$$

This yields
$$q^{\frac{3}{24}}\frac{\eta(z)^{10}\eta(2z)^{10}\eta(7z)^5}{\eta(14z)^{22}}\sum_{n=0}^\infty p_3(7n+1)q^n \equiv 
\frac{\eta(z)\eta(4z)^2\eta(14z)^{18}}{\eta(7z)^3\eta(28z)^{18}}+\frac{\eta(z)^{10}\eta(2z)^{10}\eta(7z)^2}{\eta(14z)^{22}}.$$

(It is convenient to run a brief linear computation to find a suitable form for the first term.)  Both sides are of weight zero and level dividing 28.  The order at infinity is $-10$ on the left and $-11$ on the right.  The first term on the right-hand side is of nontrivial character, so we must employ the latter bound in Theorem \ref{sturmthe}.

We then check orders at the possible cusps  $\frac{c}{d}$ with the theorems of Radu and Ligozat, and find that multiplying by $\eta(4z)^{264}$ clears all the poles.  A computation in Mathematica now quickly confirms that the right- and the left-hand side are congruent modulo $2$ up to the Sturm bound of 6336, and thus the proof of $(7,1,3)$ is complete. 

For the remaining cases of Theorem \ref{TwoTermRels}, namely $(17,5,1)$, $(19,4,1),$ and $(23,1,1)$, one can use (among other possible choices) the  $s$-vectors $(16,2,17,-34)$, $(18,2,19,-38)$, and $(22,2,23,-46)$, respectively, and then the poles may be cleared with $\eta(4z)^{3(m^2-1)}$ (here $m=17, 19$, or 23). (Note that here all modular forms involved are of trivial character.)
\end{proof} 

We mentioned earlier that all the cases of Theorems \ref{TwoTermRels} and \ref{ThreeTermRels} could in fact be proven by this method.  We provide below a list of corresponding $s$-vectors and the necessary powers of $\eta(4z)^{24}$; for the cases of Theorem \ref{TwoTermRels}, we have $N=2m$, so the divisors for the $s$-vector are $1,2,m,2m$.  For the two cases of Theorem \ref{ThreeTermRels}, we have $N=27$ and $N=125$, and therefore the divisors are $1,m,m^2,m^3$.  Given these facts, it is then a matter of verifying that all conditions are satisfied, making any necessary modulo $2$ reductions to suit the conditions of Theorem \ref{ghn}, finding the sufficient power of $\eta(4z)^{24}$ to clear the poles, determining the characters of all forms on the right-hand side, calculating the Sturm bound for the resulting forms, and finally checking congruence modulo 2 of the coefficients up to that bound to confirm the claims.\\

\begin{center}
\begin{tabular}{|c|c|c|}
\hline $(a,b,t)$ & $s$-vector & $j$ in $\eta(4z)^{24j}$ \\
\hline Any $(m,b,1)$ in Theorem \ref{TwoTermRels} & $(m-1,2,m,-2m)$ & $(m^2-1)/8$ \\
\hline $(3,2,3)$ & $(6,6,9,-18)$ & 3 \\
\hline $(5,2,3)$ & $(10,8,1,-16)$ & 6 \\
\hline $(7,1,3)$ & $(10,10,5,-22)$ & 11 \\
\hline $(5,0,5)$ & $(5,1,4,-5)$ & 4 \\
\hline $(3,0,9)$ & $(9,3,6,-9)$ & 3 \\
\hline $(3,8,3)$  in Theorem \ref{ThreeTermRels} & $(3,1,8,-9)$ & 28 \\
\hline $(5,24,1)$ in Theorem \ref{ThreeTermRels} & $(5,4,2,-10)$ & 200 \\
\hline 
\end{tabular}
\end{center}

\section{Acknowledgements} We are very grateful to an anonymous reviewer for kindly pointing out an omission in our original use of Sturm's theorem. We thank Mike Hirschhorn, Jia-Wei Kuo, and Winnie Li for helpful comments. This work will be part of the first author's Ph.D. dissertation, written at Michigan Tech under the supervision of the third author. This paper was written while the third author was partially supported by a Simons Foundation grant (\#274577).

\end{document}